\documentclass{gtpart}
\usepackage{fullpage}
\usepackage{amsmath}
\usepackage{amsfonts}
\usepackage{amssymb}
\usepackage{amsthm}
\usepackage{graphicx,xy}
\usepackage{enumerate}

\DeclareMathAlphabet{\mathpzc}{OT1}{pzc}{m}{it}

\def\cal{\mathcal}

\def\F2{\mathbb F_2}
\def\A2{{\mathcal A}_2}
\def\Z{{\mathbb Z}}
\def\N{{\mathbb N}}

\def\BC{{\rm BiCompl}}
\def\As{{\mathcal A}s}
\def\dAs{d\As}

\def\Mor{{\rm Mor}}

\def\Hom{{\rm Hom}}

\def\End{{\rm End}}
\def\Mor{{\rm Mor}}

\def\kfield{\mathbf k}

\newcommand{\ac}{\scriptstyle \text{\rm !`}}

\def\dgkvs{{\rm dg\kfield\text{-}vs}}

\def\HOM{{\textbf{Hom}}}
\def\Tw{{\rm Tw}}

\DeclareMathOperator{\hh}{HH}
\DeclareMathOperator{\h}{H}

\theoremstyle{definition}
\newtheorem{defn}{Definition}[section]
\newtheorem{prop}[defn]{Proposition}
\newtheorem{lem}[defn]{Lemma}
\newtheorem{thm}[defn]{Theorem}
\newtheorem{cor}[defn]{Corollary}
\newtheorem*{rem}{Remark}

\newtheorem*{nota}{Notation}

\xyoption{all}

\begin{document}

\title[Derived $A$-infinity algebras in an operadic context]{Derived $A$-infinity algebras in an operadic context}
\author{Muriel Livernet}
\address{Universit\'e Paris 13, Sorbonne Paris Cit\'e, LAGA, CNRS (UMR 7539), 99 avenue
  Jean-Baptiste Cl\'ement, F-93430
  Villetaneuse, France}
\email{livernet@math.univ-paris13.fr}
\author{Constanze Roitzheim}
\address{School of Mathematics, Statistics and Actuarial Science, University of Kent, Cornwallis, Canterbury, Kent, CT2 7NF, UK}
\email{c.roitzheim@kent.ac.uk}
\author{Sarah Whitehouse}
\address{School of Mathematics and Statistics, Hicks Building,
University of Sheffield, S3 7RH, England}
\email{s.whitehouse@sheffield.ac.uk}

\date{\today}
\begin{abstract}
Derived $A$-infinity algebras were developed recently by Sagave. Their
advantage over classical $A$-infinity algebras is that no projectivity
assumptions are needed to study minimal models of differential graded
algebras. We explain how derived $A$-infinity algebras can be viewed as
algebras over an operad. More specifically, we describe how this operad
arises as a resolution of the operad $\dAs$ encoding bidgas, i.e. bicomplexes with an associative multiplication. This
generalises the established result describing the operad $A_\infty$ as a
resolution of the operad $\As$ encoding associative algebras. We further
show that Sagave's definition of morphisms agrees with the infinity-morphisms of
$dA_\infty$-algebras arising from operadic machinery. We also study the operadic
homology of derived $A$-infinity algebras.
\end{abstract}
\maketitle

\setcounter{tocdepth}{2}
\tableofcontents

\section*{Introduction}

Mathematical areas in which
$A_\infty$-structures arise range from geometry, topology and representation theory to mathematical physics. One important application is to the study of differential graded algebras via $A_\infty$-structures on their homology algebras. This is the theory of minimal models established by Kadeishvili in the 1980s \cite{Kad79}. However, the results concerning minimal models all have rather restrictive projectivity assumptions.

To bypass these projectivity assumptions, Sagave recently developed the notion of derived $A_\infty$-algebras~\linebreak \cite{Sag10}. Compared to classical $A_\infty$-algebras, derived $A_\infty$-algebras are equipped with an additional grading. Using this definition one can define projective resolutions that are compatible with $A_\infty$-structures. With these, Sagave established a notion of minimal models for differential graded algebras (dgas) whose homology is not
necessarily projective.

Sagave's descriptions of derived $A_\infty$-structures are largely formula-based.
In this paper, we provide an alternative description of these structures using operads. It is
not hard to write down an operad $dA_\infty$ that encodes derived $A_\infty$-structures, but we also explain the context into which this operad fits. The category we are going to work in is the category
$\BC_v$ of bicomplexes with no horizontal differential. We will start from an operad $\dAs$ in this category encoding bidgas, that is, monoids
in bicomplexes (see Definition~\ref{def:bidga}). Our main theorem shows that derived $A_\infty$-algebras are algebras over the operad $$dA_\infty=(\dAs)_\infty= \Omega((d\As)^{\ac}).$$ This means that the operad $dA_\infty$ is a minimal model of a well-known structure.

We can summarize our main result and its relation to the classical case in the following table.
\smallskip

\begin{center}
\begin{tabular}{ccc}
underlying category&operad $\cal{O}$&$\cal{O}$-algebra\\
\hline
differential graded $\kfield$-modules&$\As$&dga\\
&$A_\infty$&$A_\infty$-algebra\\
\hline
$\BC_v$&$\dAs$&bidga\\
&$dA_\infty$&derived $A_\infty$-algebra\\
\hline
\end{tabular}
\end{center}
\medskip

We hope that this provides a useful way of thinking about derived $A_\infty$-structures.
It should allow many operadic techniques to be applied to their study and we give two
examples. Firstly, we note a simple consequence of the homotopy transfer theorem.
Secondly we develop operadic homology of derived $A_\infty$-algebras and relate this
to formality of dgas.

\bigskip

This paper is organised as follows. We start by recalling some previous results in Section \ref{sec:review}. In the first part we summarise some definitions, conventions and results about derived $A_\infty$-algebras. The second part is concerned with classical $A_\infty$-algebras. We look at the operad $\As$ encoding associative algebras and summarise how to obtain the operad $A_\infty$ as a resolution of $\As$.

In Section~\ref{sec:dAs} we generalise this to the operad $\dAs$. More precisely, this operad lives in the category of bicomplexes with trivial horizontal differential. It encodes bidgas and can be described as the composition of the operad of dual numbers and $\As$ using a distributive law. The main result of this section is computing its Koszul dual cooperad.

Section~\ref{sec:dAinfty} contains our main result. We describe the operad $dA_\infty$ encoding derived
$A_\infty$-algebras and show that it agrees with the cobar construction of the reduced Koszul dual cooperad of $\dAs$.

In Section~\ref{sec:infmorphims} we consider $\infty$-morphisms and show that they coincide with the derived $A_\infty$-morphisms
defined by Sagave. We also give an immediate application of the operadic approach, by deducing the existence
of a $dA_\infty$-algebra structure on the vertical homology of a bidga over a field
from the homotopy transfer theorem.

In Section~\ref{sec:HH}, we study the operadic homology of derived $A_\infty$-algebras. By comparing this to the
previously defined Hochschild cohomology of~\cite{RoiWhi11}, we deduce a criterion for intrinsic formality of a dga.

We conclude with a short section outlining some areas for future investigation.

\bigskip
The second author was supported by EPSRC grant EP/G051348/1.

\section{A review of known results}\label{sec:review}

Throughout this paper let $\kfield$ denote a commutative ring unless stated otherwise. All operads considered are non-symmetric.

\subsection{Derived \texorpdfstring{$A_\infty$}{A-infinity}-algebras}

We are going to recall some basic definitions and results regarding derived $A_\infty$-algebras. This is just a brief recollection; we refer
to~\cite{Sag10} and~\cite{RoiWhi11} for more details.

\bigskip
We start by considering $(\mathbb{N},\mathbb{Z})$-bigraded $\kfield$-modules
\[
A = \bigoplus\limits_{i \in \mathbb{N}, j \in \mathbb{Z}} A^j_i.
\]
The lower grading is called the \emph{horizontal degree} and the upper grading the \emph{vertical degree}. Note that the horizontal grading is homological whereas the vertical grading is cohomological. A morphism of bidegree $(u,v)$ is then a morphism of bigraded modules that lowers the horizontal degree by $u$ and raises the vertical degree by $v$. We are observing the \emph{Koszul sign rule}, that is
\[
(f \otimes g)(x\otimes y) = (-1)^{pi+qj} f(x) \otimes g(y)
\]
if $g$ has bidegree $(p,q)$ and $x$ has bidegree $(i,j)$.
Here we have adopted the grading conventions used in~\cite{RoiWhi11}.

We can now say what a derived $A_\infty$-algebra is.
\begin{defn}\cite{Sag10}
A {\it derived $A_\infty$-structure} (or {\it $dA_\infty$-structure}
for short) on an $(\mathbb{N}$,$\mathbb{Z})$-bigraded $\kfield$-module $A$
consists of $\kfield$-linear maps
\[
m_{ij}: A^{\otimes j} \longrightarrow A
\]
of bidegree $(i,2-(i+j))$ for each $i \ge 0$, $j\ge 1$,  satisfying the equations
\begin{equation}\label{dobjectequation}
\sum\limits_{\substack{ u=i+p, v=j+q-1 \\ j=1+r+t}} (-1)^{rq+t+pj}
m_{ij} (1^{\otimes r} \otimes m_{pq} \otimes 1^{\otimes t}) = 0
\end{equation}
for all $u\ge 0$ and $v \ge 1$. A {\it $dA_\infty$-algebra} is a
bigraded $\kfield$-module together with a $dA_\infty$-structure.
\end{defn}

\begin{defn}\cite{Sag10}\label{def:morphism}
A map of  $dA_\infty$-algebras from $(A, m^A)$ to $(B, m^B)$ consists of a family of
$\kfield$-module maps $f_{ij}: A^{\otimes j}\rightarrow B$ of bidegree $(i,1-i-j)$ with $i\geq 0, j\geq 1$, satisfying
\begin{equation}\label{dmapequation}
\sum\limits_{\substack{ u=i+p, v=j+q-1 \\ j=1+r+t}} (-1)^{rq+t+pj}
f_{ij} (1^{\otimes r} \otimes m^A_{pq} \otimes 1^{\otimes t}) = \\
\sum\limits_{\substack{ u=i+p_1+\cdots+p_j,\\ v=q_1+\cdots+q_j }} (-1)^{\sigma} m^{B}_{ij}(f_{p_1q_1}\otimes\cdots\otimes f_{p_jq_j}),
\end{equation}
with
    $$
    \sigma=u+\sum\limits_{k=1}^{j-1}(p_k+q_k)(j+k)+q_k(\sum\limits_{s=k+1}^j p_s+q_s).
    $$
\end{defn}

Sagave does not define composition of maps of $dA_\infty$-algebras directly in terms of this definition.
Instead this is done via a certain reformulation as maps on the reduced tensor algebra; see~\cite[4.5]{Sag10}.
It follows that $dA_\infty$-algebras form a category.

Examples of $dA_\infty$-algebras include classical $A_\infty$-algebras, which are derived $A_\infty$-algebras concentrated in horizontal degree 0. Other examples are bicomplexes and bidgas, in the sense of the following definition.

\begin{defn}
\label{def:bidga}
A \emph{bidga} is a derived $A_\infty$-algebra with $m_{ij}=0$ for $i+j \geq 3$.
A \emph{morphism of bidgas} is a morphism of derived $A_\infty$-algebras $f_{ij}$ with $f_{ij}=0$ for $i+j\geq 2$.
\end{defn}

Sagave notes that this is equivalent to saying that a bidga is a monoid in the category of bicomplexes.
\medskip

For derived $A_\infty$-algebras, the analogue of a quasi-isomorphism is called an $E_2$-equivalence. To explain this, we need to discuss
twisted chain complexes. The terminology \emph{multicomplex} is also used for a twisted chain complex.

\begin{defn}\label{def:twistedchaincx}
A \emph{twisted chain complex} $C$ is an $(\N,\Z)$-bigraded $\kfield$-module with differentials $d_i^C:C\longrightarrow C$ of  bidegree
$(i, 1-i)$ for $i\geq 0$ satisfying $$\sum_{i+p=u} (-1)^{i}d_i^Cd_p^C=0$$ for $u\geq 0$. A \emph{map of twisted chain complexes} $C\longrightarrow D$ is a family of maps $f_i:C\longrightarrow D$ of bidegree $(i,-i)$ satisfying
    $$
    \sum_{i+p=u} (-1)^{i} f_id^C_p = \sum_{i+p=u}  d^D_i f_p.
    $$
The composition of maps $f:E\to F$ and $g: F\to G$ is defined by $(gf)_u=\sum_{i+p=u}g_if_p$ and the resulting category
is denoted ${\rm{tCh}}_k$.
\end{defn}

A derived $A_\infty$-algebra has an underlying twisted chain complex, specified by the maps $m_{i1}$ for $i\geq 0$.

If $f: C\longrightarrow D$ is a map of twisted chain complexes, then $f_0$ is a $d_0$-chain map and $H_*^v(f_0)$ induces a $d_1$-chain map.

\begin{defn}\label{def:equivs}
A map $f: C\longrightarrow D$ of twisted chain complexes is an \emph{$E_1$-equivalence}
if $H_t^v(f_0)$ is an isomorphism for
all $t\in\Z$
and an \emph{$E_2$-equivalence} if $H_s^h(H_t^v(f_0))$ is an isomorphism for
all $s\in\N$, $t\in\Z$.
\end{defn}

The first main advantage of derived $A_\infty$-structures over $A_\infty$-structures is that one has a reasonable notion of a minimal model for differential graded algebras without any projectivity assumptions on the homology.

\begin{thm}\cite{Sag10}
Let $A$ be a dga over $\kfield$. Then there is a degreewise $\kfield$-projective
$dA_\infty$-algebra $E$ together with an $E_2$-equivalence $E
\longrightarrow A$ such that
\begin{itemize}
\item $E$ is minimal (i.e.~$m_{01}=0$),
\item $E$ is unique up to $E_2$-equivalence,
\item together with the differential $m_{11}$ and the multiplication
$m_{02}$, $E$ is a termwise $\kfield$-projective resolution of the graded algebra
$H^*(A)$.
\end{itemize}
\end{thm}

The second and third authors then gave the analogue of Kadeishvili's formality criterion for dgas using Hochschild cohomology. They describe derived $A_\infty$-structures in terms of a Lie algebra structure on morphisms of the underlying $\kfield$-module $A$. Then they use this Lie algebra structure to define Hochschild cohomology for a large class of derived $A_\infty$-algebras and eventually reach the following result~\cite[Theorem 4.4]{RoiWhi11}. Recall that a dga is called intrinsically formal if any other dga $A'$ such that $H^*(A)\cong H^*(A')$
as associative algebras is quasi-isomorphic to $A$.

\begin{thm}\cite{RoiWhi11}
Let $A$ be a dga and $E$ its minimal
model with $dA_\infty$-structure $m$. By $\tilde{E}$, we denote the
underlying bidga of $E$, i.e. $\tilde{E}=E$ as $\kfield$-modules together
with $dA_\infty$-structure $\tilde{m}=m_{11}+m_{02}$. If
\[
HH^{m,2-m}_{bidga}(\tilde{E},\tilde{E})=0
\quad\quad\mbox{for\ } m \ge 3,
\]
then $A$ is intrinsically formal.
\end{thm}

\subsection{The operad \texorpdfstring{$\As$}{As}}

The goal of our paper is to describe derived $A_\infty$-algebras as algebras over an operad, and to show that this operad is a minimal model of a certain Koszul operad. The operad in question is an operad called $\dAs$ (defined in Section~\ref{sec:dAs}), which is a generalisation of the operad $\As$ that encodes associative algebras. So let us recall this strategy for $\As$ itself. For this subsection only, let $\kfield$ be a field. We work in the category of (cohomologically) differential graded $\kfield$-vector spaces, denoted $\dgkvs$.
\bigskip

We will use the notation $\cal F(M)$ for the free (non-symmetric) operad generated by a collection
$M=\{M(n)\}_{n\geq 1}$ of graded $\kfield$-vector spaces. It is weight graded by the number $s$ of vertices in the planar tree representation of elements of $\cal F(M)$ and we denote by $\cal F_{(s)}(M)$ the corresponding graded $\kfield$-vector space.
 We denote by $\cal P(M,R)$ the operad defined by generators and relations,
$\cal F(M)/(R)$.  A \emph{quadratic operad} is an operad such that $R\subset \cal F_{(2)}(M)$.

\begin{defn}
The operad $\As$ in $\dgkvs$ is given by
\[
\As= \cal P(\kfield\mu, \kfield as)
\]
where $\mu$ is a binary operation concentrated in degree zero, and $as=\mu \circ_1 \mu - \mu \circ_2 \mu$. The differential is trivial.
\end{defn}

It is easy to verify that an $\As$-algebra structure on the differential graded $\kfield$-vector space $A$, i.e. a morphism of dg operads $$\As \xrightarrow{\Phi} {\End}_A,$$ endows $A$ with the structure of an associative dga, with multiplication $$\Phi(\mu): A^{\otimes 2} \longrightarrow A.$$

\begin{thm}
The operad $\As$ is a Koszul operad, i.e.~the map of operads in $\dgkvs$
\[
 \Omega(\As^{\ac}) \longrightarrow \As
\]
is a quasi-isomorphism. Furthermore, an algebra over $\Omega(\As^{\ac})$ is precisely an $A_\infty$-algebra.
\end{thm}

Here, a quasi-isomorphism of operads is a quasi-isomorphism of dg-$\kfield$-vector spaces in each arity degree. We do not recall the definitions of the Koszul dual cooperad $(-)^{\ac}$ or the cobar construction $\Omega(-)$ here. (This is going to be discussed in greater detail for our computations later). Let us just mention now that the cobar construction of a cooperad is a free graded operad endowed with a differential built from the cooperad structure, so we can think of the map above as a free resolution of the operad $\As$. This result can be proved using beautiful geometric and combinatorial methods such as the Stasheff cell complex. Unfortunately, the derived case will not be as obviously geometric.

\bigskip
Our aim is to create an analogue of the above for the derived case. The first step is to consider working in a different category - instead of differential graded $\kfield$-vector spaces, we consider a category of graded chain complexes over a commutative ring $\kfield$.

The role of $\As$ in this case is going to be played by an operad $\dAs$, which encodes bidgas rather than associative dgas.

The first goal is showing that $\dAs$ is a Koszul operad, i.e. that
\[
(\dAs)_\infty := \Omega((d\As)^{\ac}) \longrightarrow \dAs
\]
is a quasi-isomorphism of operads in an appropriate category. We are going to achieve this by ``splitting'' $\dAs$ into two parts, namely the operad of dual numbers and $\As$ itself, via a distributive law.

Secondly, we are going to compute the generators and differential of $(\dAs)_\infty$ explicitly, so we can read off that $(\dAs)_\infty$-algebras give exactly derived $A_\infty$-algebras in the sense of Sagave.

Our work will show that the operad controlling derived $A_\infty$-algebras can be seen as a free resolution of the operad encoding bidgas, in the same sense that the classical $A_\infty$-operad is a free resolution of the operad encoding associative dgas.

\section{The operad \texorpdfstring{$\dAs$}{dAs}}\label{sec:dAs}

In the first part of this section, we recall some basic notions about the Koszul dual cooperad of a given operad
and we compute the Koszul dual of $\dAs$. Further details can be found in~\cite{Fresse04}, which covers Koszul
duality for operads over a general commutative ground ring.
We also refer to the book of Loday and Vallette~\cite{LodVal12}.

We are first going to specify the category we work in. Again, let $\kfield$ be a commutative ring.

\subsection{Vertical bicomplexes and operads in vertical bicomplexes}\label{SS:verticalbicomplex}

\begin{defn}
The category of \emph{vertical bicomplexes} $\BC_v$ consists of bigraded $\kfield$-modules as above together with a vertical differential
\[
d_A: A^j_i \longrightarrow A^{j+1}_{i}
\]
of bidegree $(0,1)$. The morphisms are those morphisms of bigraded modules commuting with the vertical differential. We denote by $\Hom(A,B)$ the set of morphisms (preserving the bigrading) from $A$ to $B$.

If $c,d \in A$ have bidegree $(c_1,c_2)$ and $(d_1,d_2)$ respectively we denote by $|c||d|$ the integer
$c_1d_1+c_2d_2$.

We define a degree shift operation on $\BC_v$ as follows.
Let $A \in \BC_v$. Then $sA$ is defined as $$(sA)_i^j=A_i^{j+1}$$ with $$d_{sA}(sx)=-s(d_Ax).$$
So if $c \in A$ is of bidegree $(c_1, c_2)$, then $sc \in sA$ is of bidegree $(c_1, c_2-1)$.

This shift is compatible with the embedding of differential graded complexes
into
$\BC_v$ given by $C^l_0=C^l$ and $C^l_k=0$, if $k>0$.

The tensor product of two vertical bicomplexes $A$ and $B$ is given by
    $$
    (A\otimes B)_u^v=\bigoplus_{i+p=u,\, j+q=v}A_i^j\otimes B_p^q,
    $$
with $d_{A\otimes B}=d_A\otimes 1+1\otimes d_B:(A\otimes B)_u^v\to (A\otimes B)_u^{v+1}$.
\end{defn}

Note that $\BC_v$ is isomorphic to the category of $\mathbb{N}$-graded chain complexes of $\kfield$-modules.
\medskip

There are two other sorts of morphism that we will consider later and we introduce notation for these now.
(Various alternative choices of notation are used in the literature.)
Let  $A$ and $B$ be two vertical bicomplexes. We write $\Hom_\kfield$ for morphisms of $\kfield$-modules.
We will denote by $\Mor(A,B)$ the vertical bicomplex given by
$$\Mor(A,B)_u^{v}=\prod_{\alpha,\beta} \Hom_\kfield(A^\beta_{\alpha},B^{\beta+v}_{\alpha-u}),$$
with vertical differential given by
$\partial_{\Mor}(f)= d_Bf - (-1)^{j}f d_A$
for $f$ of bidegree $(l,j)$.

We will denote by $\HOM(A,B)$ the (cohomologically) graded complex given by
$$\HOM(A,B)^k=\prod_{\alpha,\beta} \Hom_\kfield(A^\beta_{\alpha},B^{\beta+k}_{\alpha}),$$
with the same differential as above. One has
    $$
    \Hom(A,B)=\Mor(A,B)_0^0 \qquad\text{ and }\qquad \HOM(A,B)^*=\Mor(A,B)_0^*.
    $$

\begin{defn} A \emph{collection} in $\BC_v$ is a collection $A(n)_{n\geq 1}$ of vertical bicomplexes.
We denote by $\cal C\BC_v$ the category of collections of vertical bicomplexes. This category is endowed with a monoidal structure, the plethysm given by, for any two collections $M$ and $N$,

\[
(M\circ N)(n)=\bigoplus_{k,\ l_1+\cdots+l_k=n} M(k)\otimes N(l_1)\otimes\cdots\otimes N(l_k).
\]

The unit for the plethysm is given by the collection
\[
I(n)=\begin{cases} 0, & \text{ if } n\not=1, \\
\kfield \text { concentrated in bidegree } (0,0), & \text{ if } n=1.\end{cases}
\]

Given two collections $A$ and $B$ in $\BC_v$, one can consider again the three collections
\begin{itemize}
 \item $\Hom(A,B)(n):=\{\Hom(A(n),B(n)\}_{n\geq 1}$ in the category of $\kfield$-modules,
\item  $\Mor(A,B)(n):=\{\Mor(A(n),B(n)\}_{n\geq 1}$ in the category of vertical bicomplexes and
\item  $\HOM(A,B)(n):=\{\HOM(A(n),B(n)\}_{n\geq 1}$ in the category of complexes.
\end{itemize}
\end{defn}

\begin{defn} A (non-symmetric) \emph{operad} in $\BC_v$ is a monoid in $\cal C\BC_v$.
This is the usual definition of operads in the symmetric monoidal category $(\BC_v,\otimes)$.
\end{defn}

For a vertical bicomplex $A$, the \emph{endomorphism operad} $\End_A$ is  the operad in vertical bicomplexes given by $\End_A(n)=\Mor(A^{\otimes n},A)$, where
the operad structure is given by the composition of morphisms, as usual.

\subsection{The operad \texorpdfstring{$\dAs$}{dAs}}
We now describe the operad in $\BC_v$ that encodes bidgas.

\begin{defn}The operad $\dAs$ in $\BC_v$ is defined as $\cal P(M_{\dAs},R_{\dAs})$ where
$$M_{\dAs}(n)=\begin{cases} 0, & \text{if } n>2, \\
\kfield  m_{02} \text{ concentrated in bidegree } (0,0), & \text {if } n=2, \\
\kfield  m_{11} \text{ concentrated in bidegree } (1,0), & \text {if } n=1, \end{cases}$$
and
$$R_{\dAs}=\kfield (m_{02}\circ_1 m_{02}-m_{02}\circ_2 m_{02})\oplus
\kfield m_{11}^2 \oplus \kfield (m_{11}\circ_1 m_{02}-m_{02}\circ_1 m_{11} - m_{02}\circ_2 m_{11}),$$
with trivial vertical differential.
\end{defn}

This operad is clearly quadratic.
\smallskip

The following result is now essentially a matter of definitions, but we include the details for completeness.

\begin{prop}\label{prop:bidga}
The category of $\dAs$-algebras in $\BC_v$ is isomorphic to the category of bidgas.
\end{prop}

\begin{proof}
A $\dAs$-algebra structure on a vertical bicomplex $A$ is given by a morphism of operads
\[
\theta: \dAs \longrightarrow \End_A.
\]
Since $A$ is a vertical bicomplex, it is $(\N,\Z)$-graded and comes with a vertical differential
$d_A=d^v$ of bidegree $(0,1)$. From the images of the operad generators we have morphisms
\begin{align*}
m&=\theta(m_{02}):A^{\otimes 2}\longrightarrow A,\\
d^h&=\theta(m_{11}):A\longrightarrow A,
\end{align*}
of bidegree $(0,0)$ and $(1,0)$ respectively.

The operad relations tell us precisely that $m$ is associative, that $d^h$ is a differential and that
$d^h$ is a derivation with respect to $m$. The fact that $\theta$ is a morphism of operads in $\BC_v$, and that the
differential on each $\dAs(n)$ is trivial, gives us two further relations:
\begin{align*}
\partial_{\Mor}(m)&=0,\\
\partial_{\Mor}(d^h)&=0.
\end{align*}
The first of these relations tells us that $d^v$ is a derivation with respect to $m$ and the second
that $d^vd^h-d^hd^v=0$.
This gives $A$ precisely the structure of a bidga (with exactly Sagave's sign conventions).

A morphism of $\dAs$-algebras $f:A\longrightarrow B$ is a map of vertical bicomplexes which also
commutes with $m$ and $d^h$. This is precisely a morphism of bidgas.
\end{proof}

\bigskip
Let us describe the operad $\dAs$ in a little more detail. Let $m_k$ denote any $(k-1)$-fold composite of $m_{02}$. (Because of the associativity relation, $m_k$ does not depend on the choice of composition.) Due to the ``Leibniz rule relation'' every element of $\dAs$ in arity $k$ can be written as a $\kfield$-linear combination of the elements
\[
m_k(m_{11}^{\epsilon_1},...,m_{11}^{\epsilon_k})
\]
with $\epsilon_i \in \mathbb{Z}/2$. The partial composition is given by
\begin{multline}
m_l(m_{11}^{\epsilon_1},...,m_{11}^{\epsilon_l}) \circ_i m_k(m_{11}^{\delta_1},...,m_{11}^{\delta_k} )\\
= (-1)^{\alpha}\nonumber
    \begin{cases}
        \sum\limits_{s=1}^k (-1)^{\beta}
            m_{k+l-1}(m_{11}^{\epsilon_1},...,m_{11}^{\epsilon_{i-1}},m_{11}^{\delta_1},...,m_{11}^{\delta_s + 1},...,m_{11}^{\delta_k},
            m_{11}^{\epsilon_{i+1}},...,m_{11}^{\epsilon_l}), & \text{ if } \epsilon_i=1, \\
            m_{k+l-1}(m_{11}^{\epsilon_1},...,m_{11}^{\epsilon_{i-1}},m_{11}^{\delta_1},...,m_{11}^{\delta_s },...,m_{11}^{\delta_k},
            m_{11}^{\epsilon_{i+1}},...,m_{11}^{\epsilon_l}), & \text{ if } \epsilon_i=0,
    \end{cases}
\end{multline}
where $\alpha = \left(\sum\limits_{j=i+1}^l \epsilon_j\right)\left(\sum\limits_{r=1}^k \delta_r\right)$ and
$\beta= \sum\limits_{r=1}^{s-1} \delta_r$.

We see that we have an isomorphism of bigraded $\kfield$-modules,
\[
\dAs(n) \cong \kfield[x_1,...,x_n]/(x_1^2,...,x_n^2), \qquad |x_i|=(1,0)
\]
determined by assigning the monomial $x_1^{\epsilon_1}\dots x_n^{\epsilon_n}$
to the element $m_n(m_{11}^{\epsilon_1},...,m_{11}^{\epsilon_n})$.
\medskip

Let $\cal D$ denote the operad of dual numbers in the category of vertical bicomplexes, namely
\[
\cal D = \cal P(\kfield m_{11}, \kfield m_{11}^2)
\]
with trivial differential.
\medskip

We can now reformulate the above description of $\dAs$ in terms of
plethysm and distributive laws; see~\cite[8.6]{LodVal12}.

\begin{lem}\label{lem:distributive}
The map
\[
\varphi: \cal D \circ \As \longrightarrow \As \circ \cal D
\]
determined by
\[
\varphi: m_{11}\circ_1 m_{02} \mapsto m_{02} \circ_1 m_{11} +m_{02} \circ_2 m_{11}
\]
defines a distributive law, such that the induced operad structure on $ \As \circ \cal D$
coincides with the operad $\dAs$.
\end{lem}

\begin{proof}
We adopt the notation and terminology of~\cite[8.6.3]{LodVal12}.
We define
    $$
    \varphi:\kfield m_{11}\circ_{(1)}\kfield m_{02}\longrightarrow \kfield m_{02}\circ_{(1)}\kfield m_{11}
    $$
as above. This gives a rewriting rule for the quadratic operads $\cal{D}$ and $\As$ and it is clear that
$\dAs$ is isomorphic to $\As\vee_\varphi \cal D$.
From the description of the operad $\dAs$ above, we see that the induced map
$\As\circ\cal{D}\longrightarrow \As\vee_\varphi \cal D\cong \dAs$ is an isomorphism.
So, by~\cite[Proposition 8.6.4]{LodVal12}, $\varphi$ induces a distributive law and an isomorphism
of operads $\As\circ\cal{D}\longrightarrow \As\vee_\varphi \cal D$.
\end{proof}

For $\cal P = \cal P(M,R)$ a quadratic operad, the \emph{Koszul dual cooperad} $\cal P^{\ac}$ of $\cal P$ is given by
\[
\cal P^{\ac} = {\cal C}^c(sM,s^2 R).
\]
Here ${\cal C}^c(E,R)$ denotes the  cooperad cogenerated by $E$ with corelations $R$.
(For a description see \cite[Section 7.1.4]{LodVal12}.)

There are two ways of describing the cooperad $(d\As)^{\ac}$, either by describing the distributive law $$\cal D^{\ac}\circ \As^{\ac}\rightarrow \As^{\ac}\circ \cal D^{\ac}$$ or by describing the elements of ${\cal C}^c(s(\kfield m_{11}\oplus\kfield m_{02}), s^2R_{\dAs})$ in the cofree cooperad ${\cal F}^c(s(\kfield m_{11}\oplus\kfield m_{02}))$. The first description implies that for every $n$, $(d\As)^{\ac}(n)$ is a free $\kfield$-module.

\begin{prop}\label{prop:distributivedual}
The underlying collection of the cooperad ${\dAs}^{\ac}$ is isomorphic to that of
$${\cal D}^{\ac}\circ \As^{\ac}=\kfield[\mu_{11}]\circ \As^{\ac}$$ where
$\mu_{11}$ has bidegree $(1,-1)$. Hence, as a $\kfield$-module, $(d\As)^{\ac}(n)$ is free with basis given by elements $\nu_{in}$ of bidegree $(i,1-i-n)$. These elements are in 1-to-1 correspondence with
the elements $s(m_{11})^i\circ \mu_n$ in $\cal D^{\ac}\circ\As^{\ac}$.
\end{prop}

\begin{proof}
The first part of the claim follows from Lemma~\ref{lem:distributive}, since
$\dAs\cong \As\vee_{\varphi}\cal{D}$ and by~\cite[Proposition 8.6.15]{LodVal12},
there is an isomorphism of underlying collections
$(\As\vee_{\varphi}\cal{D})^{\ac}\cong \cal{D}^{\ac}\circ\As^{\ac}$.

The cooperad structures of $\cal D^{\ac}$ and $\As^{\ac}$ are well-known and can be shown by induction with the methods used in Theorem~\ref{T:cooperad}. In arity $n$, $\As^{\ac}(n)$ is a free $\kfield$-module on the generator $\mu_n$. The element $\mu_n$ has bidegree $(0,1-n)$. The cooperad $\cal D^{\ac}$ is concentrated in arity 1. It is the free cooperad on the generator $sm_{11}$. This implies that $(d\As)^{\ac}(n)$ is free on the images $\nu_{in}$ in $(\dAs)^{\ac}(n)$ of the generators
$$(sm_{11})^i\circ\mu_n \in \cal (D^{\ac}\circ\As^{\ac})(n).$$
We can read off a generator's bidegree as
\[
|\nu_{in}|  =  i(|m_{11}| + |s|) + |\mu_n| = (i,1-i-n).
\]
\end{proof}

\begin{nota}
Let $\mathcal{C}$ be a cooperad and $c \in \mathcal{C}(n)$. We are going to describe the cocomposition
\[
\Delta: \mathcal{C} \longrightarrow \mathcal{C} \circ \mathcal{C}.
\]
We write
\[
\Delta(c)=\sum\limits_{j,|I|=n} c_j; c_I.
\]
Here, $I = (i_1,...,i_j)$ is a $j$-tuple with  $|I|=i_1+\cdots+i_j$, and
\[
c_I=c_{i_1}\otimes\cdots\otimes c_{i_j} \in \mathcal{C}^{\otimes j}.
\]
If $\mathcal{C}={\cal F}^c(V)$ is a cofree cooperad cogenerated by a collection $V$, then it has a description in terms of trees whose vertices are labelled by elements of $V$; see~\cite[5.8.7]{LodVal12}. Moreover if $V(n)$ is a free $\kfield$-module for each $n$, then so is
$\mathcal{C}(n)$, and a basis as a free $\kfield$-module is given by planar trees whose vertices are labelled by a basis of $V$. If the root of such a tree has arity $k$ and is labelled by $v$ we denote it by $v(t^1,\ldots,t^k)$ where $t^1,\ldots, t^k$ are elements of $\mathcal{C}={\cal F}^c(V)$. Remembering that
\[
\Delta(t^r)=\sum t^r_{j_r};t^r_{I_r}
\]
one obtains the formula
\begin{equation}\label{F:deltacoop}
\Delta(v(t^1,\ldots,t^k))=1;v(t^1,\ldots,t^k)+
\sum (-1)^{\sum\limits_{r=1}^{k-1} |t^r_{I_r}|(\sum\limits_{s=r+1}^k |t^s_{j_s}|)} v(t^1_{j_1},\ldots,t^k_{j_k}); t^1_{I_1}\otimes\cdots\otimes t^k_{I_k}.
\end{equation}
\end{nota}

We now compute the full structure of $(\dAs)^{\ac}$. From Proposition~\ref{prop:distributivedual} we already know the structure of its underlying bigraded $\kfield$-modules, and we can use (\ref{F:deltacoop}) to write down the cocomposition of its basis elements.

We remark that we have chosen to work directly with the cooperad $(d\As)^{\ac}$, rather than with the operad
$(\dAs)^{!}$. This is to avoid taking linear duals, which can be badly behaved over a general ground ring.

\begin{thm}\label{T:cooperad}
The cooperad $(d\As)^{\ac}$ is a sub-cooperad of ${\cal F}^c(sM_{\dAs})$ with trivial differential. Its underlying collection consists of free $\kfield$-modules with basis
$\{\mu_{ij}, i\geq 0,j\geq 1\}$
such that $\mu_{01}$ is the identity of the cooperad,
$\mu_{02}=sm_{02}$ and $\mu_{11}=sm_{11}\in {\cal F}^c(sM_{\dAs})$. The other $\mu_{ij}$ are defined inductively via
$$\begin{array}{rcll}
\mu_{i1}&=&\mu_{11}(\mu_{i-1,1}), & \text{for\ } i\geq 1, \\
\mu_{0n}&=&\sum\limits_{p+q=n}(-1)^{p(q+1)}\mu_{02}(\mu_{0p},\mu_{0q}), & \text{for\ } n\geq 2, \\
\mu_{ij}&=&\mu_{11}(\mu_{i-1,j})+\sum\limits_{r+t=i\atop{s+w=j}}(-1)^{|s\mu_{rs}||\mu_{tw}|+rw} \mu_{02}(\mu_{rs},\mu_{tw}),
                        & \text{for\ } i\geq 1, j\geq 2.
  \end{array}$$
The element $\mu_{ij}$ has bidegree $(i,1-i-j)$. These elements satisfy
\begin{equation}\label{F:Delta}
    \Delta(\mu_{uv})=\sum\limits_{i+p_1+\cdots+p_j=u\atop{q_1+\cdots+q_j=v}}
                    (-1)^{X\left((p_1,q_1), \dots, (p_j,q_j)\right)}
                        \mu_{ij};\mu_{p_1q_1}\otimes\cdots\otimes \mu_{p_j q_j},
\end{equation}
where
\begin{equation}
    \begin{aligned}\label{F:sign}
    X\left((p_1,q_1), \dots, (p_j,q_j)\right) &= & \sum\limits_{k=1}^{j-1} |s\mu_{p_kq_k}|(\sum\limits_{l=k+1}^j |\mu_{p_lq_l}|)+
            \sum\limits_{k=1}^{j-1} p_k(\sum\limits_{l=k+1}^j q_l)\\
            & = & \sum\limits_{k=1}^{j-1} \Big(  (p_k+q_k)(j+k)+ q_k\sum\limits_{l=k+1}^j (p_l + q_l)      \Big).
    \end{aligned}
\end{equation}
\end{thm}

\begin{proof}

Firstly we are going to show that those inductively defined elements form a sub-cooperad of $\cal F^c(sM_{\dAs})$.
Then we will see that this sub-cooperad contains the quadratic relations $s^2R_{\dAs}$.
Together with Proposition \ref{prop:distributivedual},  this means that it must be $(\dAs)^{\ac}$ itself.

For the first part we have to prove formula (\ref{F:Delta}), which is done by induction on $u+v$.

One has $$\Delta(\mu_{u1})=\sum\limits_{i+p=u} \mu_{i1};\mu_{p1}$$ which is proved by induction from the definition $$\mu_{u1}=\mu_{11}(\mu_{u-1,1}).$$

The case of $\Delta(\mu_{0v})$ is similar to the general case $\Delta(\mu_{uv})$, so we only prove formula (\ref{F:Delta}) for $u\geq 1, v\geq 2$.

We would like to prove that
$$\Delta(\mu_{uv})=\sum (-1)^{X(I)}
\mu_{ij};\mu_I,$$
where the sum is taken over  $i,j, I=((p_1,q_1),\ldots,(p_j,q_j))$ such that $i+\sum_k p_k=u, \sum_k q_k=v$.

\bigskip
\noindent
By formula~(\ref{F:deltacoop}) we have
\begin{equation}\label{F:sumformula}
\Delta(\mu_{uv})=\Delta\left(\mu_{11}(\mu_{u-1,v})+\sum\limits_{r+t=u, s+w=v} (-1)^{|s\mu_{rs}||\mu_{tw}|+rw} \mu_{02}(\mu_{rs},\mu_{tw})\right).
\end{equation}
We will evaluate the summands on the right hand side of the above formula separately using induction together with formula~(\ref{F:deltacoop}).

Assume that we have proved (\ref{F:Delta}) for all $\mu_{kl}$ with $k+l<u+v$. This implies that
\[
\Delta(\mu_{u-1,v})=\sum\limits (-1)^{X(I)} \mu_{i-1,j};\mu_I.
\]
Applying formula~(\ref{F:deltacoop}) allows us to relate this to $\Delta(\mu_{11}(\mu_{u-1,v}))$ with the result that
\[
\Delta(\mu_{11}(\mu_{u-1,v}))=\mu_{01};\mu_{11}(\mu_{u-1,v})+\sum\limits_{}(-1)^0(-1)^{X(I)}\mu_{11}(\mu_{i-1,j});\mu_I.
\]
Thus we have computed the first summand of (\ref{F:sumformula}). As for the second summand, the induction assumption gives us
\[
\Delta(\mu_{rs})=\sum\limits_{}(-1)^{X(I_1)}\mu_{\rho\tau};\mu_{I_1} \,\,\,\,\mbox{and}\,\,\,\,
\Delta(\mu_{tw})=\sum\limits_{}(-1)^{X(I_2)}\mu_{\gamma\delta};\mu_{I_2}
\]
with $I_1=((p_1,q_1),\ldots,(p_\tau q_\tau))$ and $I_2=((p_{\tau+1},q_{\tau+1}),\ldots,(p_j,q_j))$. Putting this in~(\ref{F:deltacoop}) gives
\begin{align*}
    \Delta(\mu_{02}(\mu_{rs},\mu_{tw}))=
            &\sum\limits_{} (-1)^{\sum\limits_{k=1}^\tau|\mu_{p_kq_k}||\mu_{\gamma\delta}|}(-1)^{X(I_1)+X(I_2)}
                      \mu_{02}(\mu_{\rho\tau},\mu_{\gamma\delta});\mu_{I_1}\otimes\mu_{I_2}\\
                      &+\mu_{01};\mu_{02}(\mu_{rs},\mu_{tw}).
\end{align*}

We will feed these computations back into (\ref{F:sumformula}) and work out the signs to obtain the desired (\ref{F:Delta}).
Let $i\geq 1$ and $j\geq 2$.
We are interested in computing the signs in front of elements of the type
$\mu_{11}(\mu_{i-1,j});\mu_I$ and of the type $\mu_{02}(\mu_{\rho\tau},\mu_{\gamma\delta});\mu_I$ where

\[
\begin{array}{rll}
\rho+\gamma &=&i, \\
\tau+\delta&=&j, \\
I&=&((p_1,q_1),\ldots,(p_j,q_j)).
\end{array}
\]

\noindent For the first type the sign is $(-1)^{X(I)}$.
For the second type the sign is of the form $(-1)^Y$ where $Y$ is computed mod 2:

\begin{align*}
 Y&=|s\mu_{rs}||\mu_{tw}|+rw + \sum\limits_{k=1}^{\tau}|\mu_{p_kq_k}||\mu_{\gamma\delta}|+ X(I_1)+X(I_2) \\
&=|s\mu_{rs}||\mu_{tw}|+rw +
\sum\limits_{k=1}^{\tau} |\mu_{p_kq_k}||\mu_{\gamma\delta}|+
\sum\limits_{k=1}^{\tau-1} |s\mu_{p_kq_k}|(\sum\limits_{l=k+1}^\tau |\mu_{p_lq_l}|)+
\sum\limits_{k=1}^{\tau-1} p_k(\sum\limits_{l=k+1}^\tau q_l)\\
&\quad+\sum\limits_{k=\tau+1}^{j-1} |s\mu_{p_kq_k}|(\sum\limits_{l=k+1}^j |\mu_{p_lq_l}|)+
\sum\limits_{k=\tau+1}^{j-1} p_k(\sum\limits_{l=k+1}^j q_l).
\end{align*}

Let us now simplify the sign $Y$.
Using the equalities
\begin{align*}
|\mu_{tw}|&=|\mu_{\gamma\delta}|+\sum\limits_{k=\tau+1}^j |\mu_{p_kq_k}|,
&\rho+\sum\limits_{k=1}^\tau p_k&=r,\\
|\mu_{rs}|&=|\mu_{\rho\tau}|+\sum\limits_{k=1}^\tau |\mu_{p_kq_k}|,
  &\sum\limits_{l=\tau+1}^{j} q_l&=w,\\
\end{align*}
one gets
\begin{align*}
 Y&=X(I)+|s\mu_{rs}||\mu_{tw}|+rw +
\sum\limits_{k=1}^{\tau} |s\mu_{p_kq_k}|(\sum\limits_{l=\tau+1}^j |\mu_{p_lq_l}|)\\
&\qquad\qquad\qquad\qquad\qquad+\sum\limits_{k=1}^{\tau} |\mu_{p_kq_k}|(|\mu_{\gamma\delta}|)+(\sum\limits_{l=1}^\tau p_k)(\sum_{l=\tau+1}^{j} q_l) \\
&=X(I)+|s\mu_{\rho\tau}||\mu_{\gamma\delta}|+(|s\mu_{\rho\tau}|+\tau|s|)(\sum\limits_{k=\tau+1}^j|
\mu_{p_kq_k}|)+\rho w\\
&=X(I)+|s\mu_{\rho\tau}||\mu_{\gamma\delta}|+\rho(\delta-w)+\rho w \\
&=X(I)+|s\mu_{\rho\tau}||\mu_{\gamma\delta}|+\rho\delta.
\end{align*}
Putting this together, we obtain a summand of the form
\[
(-1)^{X(I)} ( \mu_{11}(\mu_{i-1, j});\mu_I+\sum\limits_{\rho+\gamma=i\atop{\tau+\delta=j}}(-1)^{|s\mu_{\rho\tau}||\mu_{\gamma\delta}|+\rho\delta} \mu_{02}(\mu_{\rho\tau},\mu_{\gamma\delta});\mu_I)=(-1)^{X(I)} \mu_{ij};\mu_I,
\]
for $i\geq 1$ and $j\geq 2$.

If $j=1$, we are interested in computing the sign in front of the element of the type
$\mu_{11}(\mu_{i-1,1});\mu_{u-i,v}$ if $i\geq 1$ or in front of $\mu_{01};\mu_{uv}$ if $i=0$.
In the first case one still gets $(-1)^{X(I)}$ with $I=(u-i,v)$ as well as in the second case.

If $i=0$ and $j>1$ we are interested in computing the sign in front of the elements of the type
$\mu_{02}(\mu_{0\tau},\mu_{0\delta});\mu_I$ where
$\tau+\delta=j$
which has already been computed and coincides with the desired sign. Consequently
formula~(\ref{F:Delta}) is proved.

\bigskip
Hence the collection of $\mu_{ij}$'s forms a sub-cooperad of the free cooperad ${\cal F}^c(sM_{\dAs})$. Furthermore
it contains $s^2R_{\dAs}$, since
\[
\begin{array}{rll}
\mu_{03} &=&sm_{02}\circ_1 sm_{02}-sm_{02}\circ_2 sm_{02}, \\
\mu_{12}&=&sm_{11}\circ_1 sm_{02}-sm_{02}\circ_1sm_{11}-sm_{02}\circ_2sm_{11}, \\
\mbox{and}\,\,\, \mu_{21}&=&sm_{11}\circ_1sm_{11}.
\end{array}
\]
We also know that its $\kfield$-module structure coincides with the $\kfield$-module structure of $(d\As)^{\ac}$,
since the $\kfield$-basis elements $\mu_{in}$ are in bijection with the $\nu_{in}$ of
Proposition~\ref{prop:distributivedual}.

As a consequence, the cooperad described is the cooperad $(\dAs)^{\ac}$.
\end{proof}

\begin{cor}\label{cor:infinitesimal}
The infinitesimal cocomposition on $(d\As)^{\ac}$ is given by
\[
\Delta_{(1)}(\mu_{uv})
=\sum\limits_{i+p=u \atop{r+q+t=v, r+1+t=j}}(-1)^{r(1-p-q)+pt}\mu_{ij};1^{\otimes r}\otimes \mu_{pq}\otimes 1^{\otimes t}.
\] \qed
\end{cor}

\section{Derived \texorpdfstring{$A_\infty$}{A-infinity}-structures}\label{sec:dAinfty}

In this section we will prove our main result, Theorem~\ref{thm:main}, describing derived $A_\infty$-algebras
as algebras over the operad $(\dAs)_\infty$. Again~\cite{Fresse04} is our main reference for the
cobar construction of a cooperad over a general ground ring.
We will also interpret our description in terms of coderivations and
compare with Sagave's approach.

\subsection{The operad \texorpdfstring{$dA_\infty$}{dA-infinity}}

We would now like to encode derived $A_\infty$-algebras via an operad. Recall from Section~\ref{sec:review} that a derived $A_\infty$-structure on a bigraded module $A$ consists of morphisms
\[
m_{uv}: (A^{\otimes v})^*_* \longrightarrow A^{*+2-u-v}_{*-u}
\]
such that for $u\geq 0, v\geq 1$,
\[
\sum\limits_{\substack{u=i+p, v=j+q-1,\\ j=1+r+t}} (-1)^{rq+t+pj} m_{ij}(1^{\otimes r} \otimes m_{pq} \otimes 1^{\otimes t}) =0.
\]

If one considers $-m_{01}$ as an internal differential of $A$  the  relation reads
\begin{multline*}
(-m_{01})(m_{uv})- (-1)^{u+v} \sum_{r+t+1=v}m_{uv}(1^{\otimes r} \otimes (-m_{01}) \otimes 1^{\otimes t})=\\
(-1)^{u}\sum\limits_{\substack{u=i+p, v=j+q-1\\ j=1+r+t, (i,j)\not=(0,1),(p,q)\not=(0,1)}} (-1)^{rq+t+pj} m_{ij}(1^{\otimes r} \otimes m_{pq} \otimes 1^{\otimes t}).
\end{multline*}

\begin{defn}\label{def:dainfty}
The operad $dA_\infty$  in $\BC_v$  is defined as the free operad
    $$
    \cal F(\kfield m_{uv}:u\geq 0, v\geq 1, (u,v)\neq (0,1)),
    $$
 together with the differential

\[
\partial_\infty(m_{uv})= (-1)^{u}\sum\limits_{\substack{u=i+p, v=j+q-1,\\ j=1+r+t,\\ (i,j),(p,q)\neq(0,1)}} (-1)^{rq+t+pj} m_{ij}(1^{\otimes r} \otimes m_{pq} \otimes 1^{\otimes t}).
\]
\end{defn}

Hence it is easily verified that an algebra over the operad $dA_\infty$ in $\BC_v$ is a derived $A_\infty$-algebra in the above sense.

\bigskip
For a coaugmented cooperad $\cal C$, the \emph{cobar construction} $\Omega(\cal C)$ of $\cal C$ is the operad defined as $\cal F(s^{-1} \overline{\cal C})$, where $\overline{\cal C}$ is the cokernel of the coaugmentation, together with the differential $\partial_\Omega = d_1 + d_2$. Here, $d_2$ is induced by the infinitesimal cocomposition map $\Delta_{(1)}$ of $\cal C$ and $d_1$ is induced by the internal differential of $\cal C$ itself. Note that in our case $\cal C = d\As$, this internal differential is trivial.

We can now state the main result of our paper.
\begin{thm}\label{thm:main}
The operads $(\dAs)_\infty = \Omega((d\As)^{\ac})$ and $dA_\infty$ agree. Hence, a derived $A_\infty$-algebra is a $(\dAs)_\infty$-algebra.
\end{thm}

\begin{proof}
By definition, $\Omega((d\As)^{\ac})$ is the free operad on the shift of  $\overline{(d\As)^{\ac}}$. Let us denote its generators by
\[
\rho_{ij}=s^{-1}\mu_{ij}, \qquad\text{for\ }i\geq 0, j\geq 1, i+j\not=1.
\]
The elements $\mu_{ij}$ were described in Theorem \ref{T:cooperad}. The element $\rho_{ij}$ obviously has bidegree $(i,2-i-j)$.

\bigskip
Recall that if $\cal C$ is a coaugmented cooperad then the differential on $\Omega(\cal C)$ is obtained from $\Delta_{(1)}$
as follows.
Assume $$\Delta_{(1)}(c)=\sum c_i;1^{\otimes r}\otimes c_j\otimes 1^{\otimes t},$$
then $$\partial_\Omega (s^{-1} c)=\sum (-1)^{|s^{-1}||c_i|} s^{-1}c_i(1^{\otimes r}\otimes s^{-1}c_j\otimes 1^{\otimes t}).$$

From Corollary \ref{cor:infinitesimal} one gets

\begin{equation}\label{E:useful}
\begin{aligned}
\partial_\Omega(\rho_{uv})&=-\sum\limits_{\substack{u=i+p, v=j+q-1,\\ j=1+r+t,\\ (i,j),(p,q)\neq(0,1)}} (-1)^{r(1-p-q)+pt+i+j} \rho_{ij}(1^{\otimes r} \otimes \rho_{pq} \otimes 1^{\otimes t})\\
&=(-1)^{u}\sum\limits_{\substack{u=i+p, v=j+q-1,\\ j=1+r+t,\\ (i,j),(p,q)\neq(0,1)}} (-1)^{rq+pj+t} \rho_{ij}(1^{\otimes r} \otimes \rho_{pq} \otimes 1^{\otimes t}).
\end{aligned}
\end{equation}
This is the definition \ref{def:dainfty} of the operad $dA_\infty$.
\end{proof}

\bigskip
Recall that a quadratic operad $\cal P$ is \emph{Koszul} if the map of operads
\[
\cal P_\infty := \Omega(\cal P^{\ac}) \longrightarrow \cal P
\]
is a quasi-isomorphism.

\begin{prop}\label{prop:Koszul}
The operad $\dAs$ is Koszul. Thus, $dA_\infty$ is a minimal model of $\dAs$.
\end{prop}

\begin{proof}
We know that $\dAs = \cal D \circ \As$ by Proposition~\ref{prop:distributivedual}. The operads $\cal D$ and $\As$ are Koszul.
Using Theorem 8.6.11 of~\cite{LodVal12}, $\dAs$ is Koszul.
\end{proof}

\begin{rem}
If we do not put in the multiplication and consider just the operad $\cal{D}_\infty=\Omega\cal{D}^{\ac}$ in $\BC_v$,
we obtain an operad whose algebras are precisely the
twisted chain complexes. This can be seen either directly as a bigraded version of~\cite[10.3.17]{LodVal12}
or by tracing just the $j=1$ parts of the structure through our results.
\end{rem}

\subsection{Coderivations and Sagave's approach}

We now relate derived $A_\infty$-structures to coderivations. In the classical case, an $A_\infty$-structure on the differential graded $\kfield$-module $A$ is equivalent to a coderivation of degree $+1$ on the reduced tensor coalgebra
\[
d: \overline{\mathcal{T}}^c(sA) \longrightarrow \overline{\mathcal{T}}^c(sA) \,\,\,\,\,\mbox{such that}\,\,\,\,\, d^2=0.
\]
Sagave generalised this viewpoint to derived $A_\infty$-algebras in the following way \cite[Section 4]{Sag10}. A derived $A_\infty$-structure on the bigraded $\kfield$-module $A$ is equivalent to a coderivation of degree $+1$
\[
\xymatrix{ \overline{\mathcal{T}}^c(sA) \ar[rr]^d  \ar[d]_{\Delta} && \overline{\mathcal{T}}^c(sA) \ar[d]_{\Delta} \\
\overline{\mathcal{T}}^c (sA) \otimes \overline{\mathcal{T}}^c(sA)\ar[rr]^{d\otimes 1+1\otimes d} &&\overline{\mathcal{T}}^c(sA) \otimes \overline{\mathcal{T}}^c(sA)\\
}
\]
such that $(\mathcal{T}^c(SA),d)$ is a twisted chain complex, see Definition \ref{def:twistedchaincx}, \cite[Lemma 4.1]{Sag10}. The definition of a differential of a twisted cochain complex differs from the condition $d^2=0$ by signs.

\bigskip
Our approach varies from this. In the setting of associative algebras in dg-$\kfield$-modules, one has
    $$
    \As^{\ac}(A)=\overline{\cal{T}}^c(sA).
    $$
However, $(d\As)^{\ac}(A)$ is \emph{not} given by $\overline{\cal{T}}^c(sA)$ in the derived setting - we showed its structure in Theorem \ref{T:cooperad}.

\bigskip
So in our setting, a derived $A_\infty$-structure on the vertical bicomplex $A$ is given by a coderivation of degree $+1$
\[
\xymatrix{ (d\As)^{\ac}(A) \ar[d]_{\Delta_{(1)}} \ar[rr]^{d} && (d\As)^{\ac}(A) \ar[d]_{\Delta_{(1)}} \\
((d\As)^{\ac} \circ_{(1)} (d\As)^{\ac})(A) \ar[rr]^{d\circ_{(1)}1+1\circ_{(1)} d} && ((d\As)^{\ac} \circ_{(1)} (d\As)^{\ac})(A) \\
}
\]
such that $d^2=0$. Comparing those two equivalent conditions we see the following. Sagave's description has the advantage of a much easier coalgebra structure while the complexity of the derived $A_\infty$-structure is encoded in the more complicated condition that a coderivation has to satisfy. In our description, a coderivation has to satisfy the relatively simple condition $d^2=0$ while the complexity lies in the more complicated coalgebra structure.

\section{Infinity morphisms and an application}\label{sec:infmorphims}

The main purpose of this section is to describe $\infty$-morphisms of $(d\As)_\infty=dA_\infty$-algebras, and to prove that they coincide with the derived $A_\infty$-morphisms defined by Sagave. At the end of the section, we give an application of the homotopy transfer theorem.

\subsection{Infinity morphisms}

Using the language of operads, the natural notion of morphism between two $dA_\infty$-algebras $A$ and $B$ is
a map $f:A\rightarrow B$ respecting the algebra structure. This is the notion of a strict morphism. However, in the context
of $\mathcal P_\infty$-algebras where $\mathcal P$ is a Koszul operad, there is also a
more general notion of $\infty$-morphism, which is more relevant to the homotopy
theory of $P_\infty$-algebras; see, for example,~\cite[Section 10.2]{LodVal12}. In the case of $A_\infty$-algebras, this gives rise to the usual notion of $A_\infty$-morphism
between two $A_\infty$-algebras $A$ and $B$ and this can be formulated as a morphism of differential graded coalgebras
between the bar constructions of $A$ and $B$.

As seen at the end of the previous section, a $dA_\infty$-structure $m$ on the vertical bicomplex $A$ is equivalent to a square-zero coderivation $D_m$ of degree $+1$ on the $(d\As)^{\ac}$-coalgebra $(d\As)^{\ac}(A)$.  This coalgebra corresponds to the bar construction for $A_\infty$-algebras in our framework.
This lends itself to the following definition.

\begin{defn}
Let $(A,m)$ and $(B,m')$ be $dA_\infty$-algebras. An \emph{$\infty$-morphism of $dA_\infty$-algebras} is a morphism
\[
F: ((d\As)^{\ac}(A), D_m) \longrightarrow ((d\As)^{\ac}(B),D_{m'})
\]
of $(d\As)^{\ac}$-coalgebras.
\end{defn}

We will interpret this definition in terms of twisting morphisms, but
first, we give a recollection of some facts based on the book of Loday and Vallette,
adapted to the category of vertical bicomplexes. We will need these as a basis for our computation.

\begin{defn}\label{def:convolution} Let $(\cal C, d_{\cal C})$ be a cooperad and $(\cal P, d_{\cal P})$ an operad in vertical bicomplexes. Following the notation of Section~\ref{SS:verticalbicomplex}, we consider the collection in complexes $\HOM(\cal C,\cal P)$. It is a differential graded operad called the \emph{convolution operad}.

There is an operation $\star$ on $\HOM(\cal C,\cal P)$ defined by
    \[
    f \star g: \cal C \xrightarrow{\Delta_{(1)}} \cal C \circ_{(1)} \cal C \xrightarrow{f \circ_{(1)} g}
        \cal P \circ_{(1)} \cal P         \xrightarrow{\gamma_{(1)}} \cal P,
    \]
where $\Delta_{(1)}$ and $\gamma_{(1)}$ are respectively the infinitesimal cocomposition and composition maps.
As in~\cite[6.4.4]{LodVal12}, this determines the structure of a differential graded pre-Lie algebra
on $\prod_n\HOM(\cal C, \cal P)(n)$.
The associated differential graded Lie algebra is called the \emph{convolution Lie algebra}.
\end{defn}

\begin{defn}
A \emph{twisting morphism} is an element $\alpha$ of degree $1$ in the complex $\HOM(\cal C, \cal P)$ satisfying the Maurer-Cartan equation
\[
\partial(\alpha) + \alpha \star\alpha = 0.
\]
We denote the set of twisting morphisms by $\Tw(\cal C, \cal P)$.
\end{defn}

By construction, the cobar construction $\Omega$ satisfies
\[
\Hom_{\BC_v-op}(\Omega(\cal C), \cal P) \cong \Tw(\cal C, \cal P),
\]
where the left-hand side means morphisms of operads in vertical bicomplexes.
This means that a $dA_\infty$-structure $m$ on the vertical bicomplex $A$, that is,  a square-zero coderivation $D_m$ of degree $+1$ on the $(d\As)^{\ac}$-coalgebra $(d\As)^{\ac}(A)$ as seen  at the end of the previous section, is equivalent to a twisting morphism
\[
\varphi_m \in \Tw( (d\As)^{\ac}, \End_A).
\]

\bigskip
Let $A$ and $B$ be vertical bicomplexes, and let $\End^A_B$, a collection in vertical bicomplexes, be  given by
\[
{\End}^A_B(n)=\Mor(A^{\otimes n}, B).
\]
The vertical differential is given by
\[
\partial(f)= d_B f -(-1)^j \sum\limits_{v=0}^{n-1} f(1^{\otimes v} \otimes d_A \otimes 1^{n-v-1})
\]
for $f$ in arity $n$ and bidegree $(i,j)$.

For $f \in \HOM( (d\As)^{\ac}, \End^A_B)$ and $\varphi \in \HOM((d\As)^{\ac}, \End_A)$, the map $f\ast  \varphi$ is given by the composite
\[
f \ast \varphi: (d\As)^{\ac} \xrightarrow{\Delta_{(1)}} (d\As)^{\ac} \circ_{(1)} (d\As)^{\ac} \xrightarrow{f \circ_{(1)} \varphi} \End^A_B \circ_{(1)} \End_A \xrightarrow{\rho} \End^A_B
\]
where $\rho$ is induced by the composition of maps. Similarly, for $\psi \in \HOM((d\As)^{\ac}, \End_B)$ and $f$ as above, $\psi  \circledast f$ is given by
\[
\psi \circledast f: (d\As)^{\ac} \xrightarrow{\Delta} (d\As)^{\ac} \circ (d\As)^{\ac} \xrightarrow{\psi \circ f} \End_B \circ \End^A_B \xrightarrow{\lambda} \End^A_B
\]
where $\lambda$ is given by composition of maps.

\bigskip
Now let $$\varphi_{m^A} \in \Tw((d\As)^{\ac}, \End_A) \,\,\,\mbox{and}\,\,\, \varphi_{m^B} \in \Tw((d\As)^{\ac}, \End_B)$$ be $dA_\infty$-structures on the vertical bicomplexes $A$ and $B$ respectively. By \cite[Theorem 10.2.6]{LodVal12}, an $\infty$-morphism
\[
F: (d\As)^{\ac}(A) \longrightarrow (d\As)^{\ac}(B)
\]
of $dA_\infty$-algebras is equivalent to an element $f \in \HOM( (d\As)^{\ac}, \End^A_B)$ of degree $0$ such that
\[
f \ast \varphi_{m^A} - \varphi_{m^B} \circledast f = \partial(f).
\]

(note that the vertical bicomplex $(d\As)^{\ac}(n)$ has trivial differential).
Taking this into account we arrive at the following.

\begin{thm}
An $\infty$-morphism $f: A \longrightarrow B$ of $dA_\infty$-algebras is
a morphism of derived $A_\infty$-algebras as defined by Sagave,
that is,
a collection of maps
\[
f_{uv}: A^{\otimes v} \longrightarrow B
\]
of bidegree $(u, 1-u-v)$ satisfying equation ~(\ref{dmapequation}) of Definition~\ref{def:morphism}.
\end{thm}

\begin{proof}
Assume that $f: (d\As)^{\ac} \longrightarrow \End^A_B$ satisfies
\[
f \ast \varphi_{m^A} - \varphi_{m^B} \circledast f = \partial(f).
\]
We know the structure of $(d\As)^{\ac}$ from Theorem \ref{T:cooperad}. The underlying $\kfield$-module of $(d\As)^{\ac}$ is free on generators $\mu_{uv}$ of bidegree $(u,1-u-v)$. Write
\[
f_{uv} := f(\mu_{uv})
\]
and recall that $\varphi_{m^A}(\mu_{ij})=m^A_{ij}$ and $\varphi_{m^B}(\mu_{ij})=m^B_{ij}$.

Using the formulas given by Theorem \ref{T:cooperad}, Corollary \ref{cor:infinitesimal} and because $\varphi_{m^A}$ is of bidegree $(0,1)$ we obtain
\begin{align*}
(f \ast \varphi_{m^A})(\mu_{uv})=&\sum\limits_{u=i+p \atop{v=j+q-1, j=r+t+1}} (-1)^{r(1-p-q)+pt+1+i+j} f_{ij}(1^{\otimes r} \otimes m^A_{pq} \otimes 1^{\otimes t})\\
=&\sum\limits_{u=i+p \atop{v=j+q-1, j=r+t+1}}(-1)^{rq+pj+t+u} f_{ij}(1^{\otimes r} \otimes m^A_{pq} \otimes 1^{\otimes t})
\end{align*}
and
\[
(\varphi_{m^B} \circledast f)(\mu_{uv})=\sum\limits_{} (-1)^X m_{ij}^B(f_{p_1q_1} \otimes \cdots \otimes f_{p_j q_j})
\]
where
\[
X=X((p_1,q_1),...,(p_j,q_j)) = \sum\limits_{k=1}^{j-1} \Big(  (p_k+q_k)(j+k)+ q_k\sum\limits_{l=k+1}^j (p_l + q_l)\Big).
\]
Also,
\[
\partial_{End}(f)(\mu_{uv})= d_B f_{uv} -(-1)^{1+u+v} \sum\limits_{l=0}^{v-1} f_{uv}(1^{\otimes l} \otimes d_A \otimes 1^{v-l-1}).
\]
With $d_A=m_{01}^A$ and $d_B=m_{01}^B$, this equals
\[
\partial_{End}(f)(\mu_{uv})= m_{01}^B(f_{uv}) -(-1)^{1+u+v} \sum\limits_{l=0}^{v-1} f_{uv}(1^{\otimes l} \otimes m^A_{01} \otimes 1^{v-l-1}).
\]
Putting this together, we arrive at
\[
(-1)^u\sum\limits_{u=i+p \atop{v=j+q-1, 1+r+t=j}}(-1)^{rq+t+pj} f_{ij}(1^{\otimes r} \otimes m^A_{pq} \otimes 1^{\otimes t}) = \sum\limits_{} (-1)^u(-1)^\sigma m_{ij}^B(f_{p_1q_1} \otimes \cdots \otimes f_{p_j q_j})
\]
which is exactly formula~(\ref{dmapequation}) of Sagave's definition.
\end{proof}

\subsection{The homotopy transfer theorem for \texorpdfstring{$\dAs$}{dAs} }

As an immediate application of our operadic description, we can apply the homotopy
transfer theorem; see~\cite[Section 10.3]{LodVal12}. To do so, we will need to now work over a ground field. Although
this takes us out of the context which motivated the introduction of
derived $A_\infty$-algebras, it nonetheless gives us a new family of examples.
\smallskip

Let $\mathcal{P}$ be a Koszul operad, $W$ a $\mathcal{P}_\infty$-algebra and $V$ a homotopy retract of $W$. Recall
that a $\mathcal{P}_\infty$-structure on $W$ is equivalent to an element
$\varphi \in \Tw(\mathcal{P}^{\ac}, \End_W)$. The homotopy transfer
theorem~\cite[Theorem 10.3.6]{LodVal12} says that the homotopy retract
$V$ can be given a $\mathcal{P}_\infty$-structure by the twisting morphism given by the following composite
\[
\mathcal{P}^{\ac} \xrightarrow{\Delta} \mathcal{F}^c(\bar{\mathcal{P}}^{\ac}) \xrightarrow{\mathcal{F}^c(s\varphi) }
 \mathcal{F}^c(s\End_W) \xrightarrow{\Psi} \End_V.
\]
(The map $\Delta$ is the coproduct map defined in \cite[5.8.12]{LodVal12}.)
Moreover there is a standard way to interpret this formula in terms of the combinatorics
of trees.
\medskip

We adopt the usual notation for this setting: we have the inclusion $i:V\to W$ and projection $p:W\to V$
such that $pi$ is the identity on $V$,
and a homotopy $h:W\to W$ between $ip$ and the identity on $W$, $1_W- ip = d_Wh + hd_W$.

As a special case, we consider $\mathcal{P}=\dAs$ and we let $V=A$ be a bidga over a field.
The vertical homology $W=H^v(A)$ of $A$ is a homotopy retract and
we therefore obtain a derived $A_\infty$-algebra structure on this. Write $d_h=m_{11}$
for the horizontal differential and $m=m_{02}$ for the multiplication.
Making the transferred structure explicit for this special case yields the following.

\begin{prop}
There is a derived $A_\infty$-algebra structure on the vertical homology $H^v(A)$ of a bidga $A$
over a field, which can be described as follows.
We obtain $m_{ij}$ as a (suitably signed) sum over the maps corresponding to planar trees with $j$ leaves,
where each vertex has been assigned a weight of either $2$ or $3$, and the number of vertices of weight $2$ is $i$.
The procedure for assigning a map to such a tree is as follows.
We adorn the trees with the map $i$ on the leaves, the map $p$ at the root and the map $h$ on
internal edges. On vertices, we put the multiplication $m$ at every vertex of weight $3$
and the horizontal differential $d_h$ at every vertex of weight $2$. \qed
\end{prop}

This construction specializes to the $A_\infty$-case which involves binary trees with no vertices of degree $2$.
That is, we recover the expected $A_\infty$-algebra structure on the part concentrated in degrees $(0,j)$;
see~\cite[9.4.4, 10.3.8]{LodVal12}.

The signs can be calculated recursively from
the explicit signs appearing in the formula~(\ref{F:Delta}) for $\Delta$.

\section{Operadic and Hochschild cohomology}\label{sec:HH}

In this section, we compute the tangent complex of a derived $A_\infty$-algebra $A$, define the Hochschild cohomology
of $A$ and make the link with the formality theorem of~\cite{RoiWhi11}. Hochschild cohomology has previously only
been defined, in~\cite{RoiWhi11}, for a special class of derived $A_\infty$-algebras, the ``orthogonal'' ones.

Given a vertical bicomplex $A$,  the trigraded $\kfield$-module $C_*^{*,*}(A,A)$ is defined by
    \[
    C_k^{n,i}(A,A)=\Mor(A^{\otimes n},A)_k^i.
    \]
We will describe a graded Lie structure on $CH^{*+1}(A,A)$, where the grading is the total grading
    \[
    CH^N(A,A)=\prod\limits_{n\geq 1}\ \prod\limits_{k,j| k+j+n=N} C_k^{n,j}(A,A),
    \]
that is, an element in $ C_k^{n,j}(A,A)$ has {\it total degree} $j+k+n$.

\subsection{Lie structures}

Let us make explicit Definition~\ref{def:convolution} for the differential graded pre-Lie structure on
$\prod_n\HOM((d\As)^{\ac}, \End_A)(n)$. From Corollary \ref{cor:infinitesimal}, knowing the infinitesimal cocomposition on $(d\As)^{\ac}$,
the $\star$ operation on $\HOM((d\As)^{\ac}, \End_A)$ is given by
\begin{equation}\label{eq:prelie}
(f \star g)(\mu_{uv})= \sum\limits_{j=1+r+t, u=i+p, v=r+q+t} (-1)^{r(1+p+q)+pt+|g||\mu_{ij}|} f(\mu_{ij})(1^{\otimes r} \otimes g(\mu_{pq}) \otimes 1^{\otimes t}),
\end{equation}
where $|g|$ denotes the vertical grading.
\smallskip

For every $N$, there is a bijection
\[
\Phi=\prod_n\Phi_n: \prod_n\HOM((d\As)^{\ac}, \End_A)(n)^N \longrightarrow \prod_n\prod\limits_{u}C^{n,N+1-n-u}_u(A,A)
\]
where $\Phi_n: \HOM((d\As)^{\ac}, \End_A)(n)^N \longrightarrow \prod\limits_{u}C^{n,N+1-n-u}_u(A,A)$
is given by evaluation:
\[
\Phi_n(f_n) = \prod\limits_{u}f_n(\mu_{un}).
\]
The unique preimage of a family $(G_n)_n$, where $G_n=(G_u^{n,N+1-n-u})_u$, is given by the family $g=(g_n)_n=(\Phi_n^{-1}(G_n))_n$ in degree $N$ defined via
    \[
    g_n(\mu_{un}) = G_u^{n,N+1-n-u}.
    \]
We can now transport the pre-Lie structure on $\prod_n\HOM((d\As)^{\ac}, \End_A)(n)$ to $CH^{*+1}(A,A)$ as follows:
let $F=(F_n)_{n\geq 1}$ be of total degree $N+1$ and let
$G=(G_m) _{m\geq 1}$ be of total degree $M+1$. There are unique families $f=(f_n)_n, g=(g_m)_m$ of degree $N$ and $M$ respectively such that
$F=\Phi(f)$ and $G=\Phi(g)$. Then
    \[
    F \star G := \Phi(f\star g).
    \]
Note that the total degree of $F\star G$ is $N+M+1$.
Hence the pre-Lie product decreases the total degree by one. That is, this pre-Lie product endows $CH^{*+1}(A,A)$ with the structure of a graded pre-Lie algebra.

Naturally, this gives rise to a graded Lie algebra structure on $CH^{*+1}(A,A)$ via
\[
[F,G] = F\star G - (-1)^{(N+1)(M+1)} G \star F.
\]

\bigskip
Let us now compare the pre-Lie structure above with the  pre-Lie structure on $C^{*,*}_*(A,A)$ built in \cite{RoiWhi11}. Let
$\mathfrak{f} \in C^{n,i}_k(A,A)$ and $\mathfrak{g} \in C^{m,j}_l(A,A)$. Then
\[
\mathfrak{f}=f_n(\mu_{kn})\,\,\,\mbox{with}\,\,\, |f_n|=n+i+k-1
\]
and
\[
\mathfrak{g}=g_m(\mu_{lm})\,\,\,\mbox{with}\,\,\, |g_m|=m+j+l-1.
\]
Putting this into formula (\ref{eq:prelie}) yields
\[
\mathfrak{f} \star \mathfrak{g}= \sum\limits_{r=0}^{n-1} (-1)^{(n+1)(m+1)+r(m+1)+j(n+1)+k(m+j+l+1)} \mathfrak{f}(1^{\otimes r} \otimes \mathfrak{g} \otimes 1^{\otimes n-r-1}) \in C^{n+m-1,i+j}_{k+l}.
\]
Hence we can see that the sign in this formula differs from the sign in the other pre-Lie algebra structure $\mathfrak{f} \circ_{RW} \mathfrak{g}$
given in~\cite[Definition 2.11]{RoiWhi11} by the sign $(-1)^{k(m+j+l+1)}$.
\bigskip

We can read off the following.
\begin{lem}\label{lem:squarezero}
Let $m \in CH^2(A,A)$. Then $m$ defines a $dA_\infty$-structure on $A$ if and only if $m \star m=0$. \qed
\end{lem}

\subsection{Hochschild cohomology}

We now use this new Lie structure to define another notion of Hochschild cohomology of
derived $A_\infty$-algebras. This definition differs from that
constructed in~\cite{RoiWhi11} by the different signs in the Lie structure, as
explained above. It has the advantage that it applies to all $dA_\infty$-algebras
rather than just the ``orthogonal'' ones.

\begin{defn}\label{def:hochschild}
Let $(A,m)$ be a $dA_\infty$-algebra. Then the \emph{Hochschild cohomology of $A$} is defined as
\[
\hh^*(A,A) := \h^*(CH(A,A), [m,-] ).
\]
\end{defn}

The morphism
\[
[m,-]: CH^*(A,A) \longrightarrow CH^*(A,A)
\]
is indeed a differential. Since $m$ has total degree 2 and
$[-,-]$ has total degree $-1$,
it raises degree by 1. By~\cite[Lemma 1.10]{Liv11} (with respect to the pre-Lie product $\circ$), one has
$[m,[m,-]]= [m\star m,-]$,
and the right-hand side vanishes because of Lemma \ref{lem:squarezero}.

In the case of $(A,m)$ being an associative algebra, this definition recovers the classical definition of Hochschild cohomology of associative algebras.

\begin{rem}
 Because of the bijection $\Phi$ the complex computing the Hochschild cohomology of $A$ coincides
 with the operadic cohomology. Recall that given a
$\cal P$-algebra $A$, its operadic cohomology with coefficients in itself is
$H^*(\HOM(\cal P^{\ac}(A),A),\partial_\pi)$ where $\pi$ depends on the twisting cochain defining the structure on $A$.

As an example,
when $A$ is a bidga with $m=m_{11}+m_{02}$, i.e. if $A$ is a bidga with trivial horizontal differential, the external grading is preserved by both bracketing with $m_{11}$ and $m_{02}$. Hence we can, as in~\cite[Section 3.1]{RoiWhi11}, consider bigraded Hochschild cohomology
\begin{equation}
\hh^{s,r}(A,A) = \h^s(\prod\limits_n C^{n,r}_{*-n}(A,A), [m,-]). \nonumber
\end{equation}
We denote this special case by $\hh^{*,*}_{bidga}(A,A)$. It corresponds to the operadic cohomology with respect to the operad $\dAs$.

When $\cal P$ is a Koszul operad, given a $\cal P_\infty$-algebra, one can still define its operadic cohomology as the homology of the complex
\begin{equation}\label{F:hhcomplex}
(\HOM(\cal P^{\ac}(A),A),\partial_\pi),
\end{equation}
where $\pi$ represents the twisting cochain associated to the $\cal P_\infty$-structure on $A$.

If $A$ is a derived $A_\infty$-algebra, the complex (\ref{F:hhcomplex}) is exactly the complex of Definition~\ref{def:hochschild}. That is, operadic cohomology for derived $A_\infty$-algebras is Hochschild cohomology as defined at the beginning of the subsection.

Note however, that in order to identify this cohomology theory with the Andr\'e-Quillen cohomology of derived $A_\infty$-algebras as in~\cite[Proposition 12.4.11]{LodVal12} one needs to assume that $A$ is bounded below for the vertical grading and is free as a $\kfield$-module.
\end{rem}

This more compact definition of Hochschild cohomology has some structural advantage over $\hh^*_{RW}$, the Hochschild cohomology defined in \cite{RoiWhi11}.
In particular, we see that the Lie bracket $[-,-]$ on $CH^*(A,A)$ induces a Lie bracket on $$\hh^*(A,A)=\h^*(CH^*(A,A),D=[m,-]).$$ This is the case because $D$ is an inner derivation with respect to $[-,-]$ due to the graded Jacobi identity. Hence, the bracket of two cycles is again a cycle, and the bracket of a boundary and a cycle is a boundary.

\begin{prop}
The (shifted) Hochschild cohomology of a $dA_\infty$-algebra $\hh^{*+1}(A,A)$ has the structure of a graded Lie algebra. \qed
\end{prop}

\subsection{Uniqueness and formality}

\begin{defn}
Let $A$ be a bidga with $m_{01}=0, \partial = m_{11}, \mu=m_{02}$. Then
\[
a = \sum\limits_{i,j} a_{ij}, \,\,\, a_{ij} \in C^{j,2-i-j}_i(A,A), \,\, i+j \ge 3
\]
is a \emph{twisting cochain} if $\partial + \mu + a$ is a derived $A_\infty$-structure.
\end{defn}

One can read off the following result immediately.
\begin{lem}
The element $a$ is a twisting cochain if and only if
\[
-D(a)=a\star a
\]
for $D=[\partial+\mu,-]$. \qed
\end{lem}

The above is the \emph{Maurer-Cartan formula}.
\smallskip

A key step in the obstruction theory leading to uniqueness of $dA_\infty$-structures is perturbing an existing twisting cochain by an element $b$ of total degree 1. Roughly speaking, this new perturbed $dA_\infty$-structure satisfies the following- it equals the existing $dA_\infty$-structure below a certain bidegree, is modified using $b$ in this bidegree and $E_2$-equivalent to the ``old'' $dA_\infty$-structure.
 This has been shown in detail in \cite[Lemma 3.6]{RoiWhi11}, but we verify briefly that this also works with our new Lie bracket.

\begin{lem}\label{lem:perturb}
Let $A$ be a bidga with multiplication $\mu$, horizontal differential $\partial$ and trivial vertical differential. Let $a$ be a twisting cochain. Let either
\begin{description}
\item[(A)] $b \in C^{n-1,2-(n+k)}_k(A,A)$ for some $k, n$ such that $k+n \ge 3$, satisfying
$[\partial,b]=0$
\end{description}
or
\begin{description}
\item[(B)] $b \in C^{n,2-(n+k)}_{k-1}(A,A)$, for some $k, n$ with $k+n \ge 3$, satisfying
$[\mu,b]=0$.
\end{description}

Then there is a twisting cochain $\overline{a}$ satisfying
\begin{itemize}
\item the $dA_\infty$-structures $\partial + \mu + a$ and $\overline{m} = \partial + \mu + \overline{a}$ are
$E_2$-equivalent,
\item $\overline{a}_{uv} = a_{uv}$ for $u<k$ or $v<n-1$ or $(u,v)=(k,n-1)$ in case {\bf (A)} and for $u<k-1$ or $v<n$
or $(u,v)=(k-1,n)$ in case {\bf (B)},
\item $\overline{a}_{kn} = a_{kn} - [\mu,b]$ in
case {\bf (A)},
\item $\overline{a}_{kn}= a_{kn} - [\partial,b] $ in case {\bf (B)}.
\end{itemize}

\end{lem}

\begin{proof}
A quick check of the signs in both Lie brackets shows that
\[
[\partial, b]_{RW} = [\partial, b] \,\,\,\mbox{and}\,\,\, [\mu,b]_{RW}=[\mu,b].
\]
Hence this is identical to \cite[Lemma 3.6]{RoiWhi11}, where the $\overline{a}_{uv}$ are constructed inductively.
\end{proof}

We can now proceed to our uniqueness theorem, which has been shown in the context of $[-,-]_{RW}$ and $\hh^{*,*}_{RW}$ in \cite[Theorem 3.7]{RoiWhi11}.

\begin{thm}
Let $A$ be a bidga with multiplication $\mu$, horizontal differential $\partial$ and trivial vertical differential. If
\[
\hh^{r,2-r}_{bidga}(A,A)=0 \,\,\,\mbox{for}\,\,\, r\ge 3,
\]
then every $dA_\infty$-structure on $A$ with $m_{01} =0$, $m_{11}= \partial$ and $m_{02}=\mu$ is $E_2$-equivalent to the trivial one.
\end{thm}

\begin{proof}

Let $m$ be a $dA_\infty$-structure on $A$ as given in the statement. We want to show that it is equivalent to the $dA_\infty$-structure $\partial + \mu$. We can write $m=\partial + \mu + a$ with $a$ a twisting cochain.

\medskip
We look at $a_{kn}$, $k+n=t\ge 3$. We show that $m$ is equivalent to a $dA_\infty$-structure $\bar{m}= \partial + \mu + \bar{a}$ with $\bar{a}_{kn}=0$ for fixed $t$ by induction on $k$.

\medskip
To start this induction we assume that
\[
a_{ij}=0\,\,\,\mbox{for}\,\,\, i+j<t \,\,\,\mbox{and for}\,\,\, i+j=t, \mbox{if}\,\,\, i<k.
\]
The new equivalent $dA_\infty$-structure $\overline{m}$ will also satisfy
\[
\bar{a}_{ij}=a_{ij}=0 \,\,\,\mbox{for}\,\,\, i+j<t \,\,\,\mbox{and for}\,\,\, i+j=t, \mbox{if}\,\,\, i<k
\]
as well as further
\[
\bar{a}_{kn}=0.
\]
So to construct $\overline{m}$, we ``kill'' $a_{kn}$ but leave the trivial lower degree $a_{ij}$ invariant.

\medskip
Since $a$ is a twisting cochain, it satisfies the Maurer-Cartan formula
\[
-D(a)= a\star a.
\]
However, an argument similar to \cite[Theorem 3.7]{RoiWhi11} shows that this implies $D(a_{kn})=0$ for degree reasons. Hence $a_{kn}$ is a cycle and gives us a class
\[
[a_{kn}] \in \hh^{k+n,2-k-n}_{bidga}(A,A)
\]
in the Hochschild cohomology of $A$. This cohomology group has been assumed to be zero, hence $a_{kn}$ must be a boundary too. Thus, there is a $b$ of total degree 1 with $D(b)=a_{kn}$. For degree reasons, this $b$ has to be of the form
\[
b = b_0 + b_1, \,\,\, b_0 \in C^{n,2-n-k}_{k-1}(A,A), \,\,\, b_1 \in C^{n-1,2-n-k}_k(A,A)
\]
with
\[
[\mu, b_0] = 0 \,\,\,\,\mbox{and}\,\,\,\, [\partial, b_1]=0,
\]
meaning that
\[
D(b) = D(b_0 + b_1)= [\mu, b_1] + [\partial, b_0].
\]
Then, just as in the proof of~\cite[Theorem 3.7]{RoiWhi11},
applying Lemma~\ref{lem:perturb} to $b_1$ yields a $dA_\infty$-structure $\overline{m}=\partial+\mu+\overline{a}$ with
\[
\bar{a}_{kn}=a_{kn} - [\mu, b_1] - [\partial, b_0] = a_{kn} -D(b)=0.
\]
\end{proof}

It was shown in \cite[Section 4]{RoiWhi11} that $\hh^{*,*}_{RW}(A,A)$ is invariant under $E_2$-equivalences. Since this argument is independent of choice of signs in the Lie bracket, it also holds for our $\hh^{*,*}_{bidga}(A,A)$. Hence we can now give a criterion for intrinsic formality of a dga. (Recall that a dga $A$ is intrinsically formal if for any other dga $B$ with $\h^*(A) \cong \h^*(B)$ as associative algebras, $A$ and $B$ are quasi-isomorphic.)

\begin{cor}
Let $A$ be a dga and $E$ its minimal
model with $dA_\infty$-structure $m$. By $\tilde{E}$, we denote the
underlying bidga of $E$, i.e. $\tilde{E}=E$ as $k$-modules together
with $dA_\infty$-structure $\tilde{m}=m_{11}+m_{02}$. If
\[
\hh^{m,2-m}_{bidga}(\tilde{E},\tilde{E})=0
\quad\quad\mbox{for}\ m \ge 3,
\]
then $A$ is intrinsically formal.\qed
\end{cor}

\section{Directions for further work}\label{sec:further}

In this paper we have given an operadic perspective on derived $A_\infty$-structures, allowing us to view
derived $A_\infty$-algebras as algebras over an operad. By results of various authors~\cite{Fresse09, Har10,Mur11}, it follows from
our description that there is a
model category structure on derived $A_\infty$-algebras such that the weak equivalences are the $E_1$-equivalences
(see Definition~\ref{def:equivs}). However, we do not expect this model structure to be homotopically meaningful.
Indeed, in order to view Sagave's minimal models as some kind of cofibrant replacement, one would need a model
structure in which the weak equivalences are the $E_2$-equivalences. Producing such a model structure will involve
a change of underlying category, probably to the category of twisted chain complexes. One would then need a suitable
model structure on this underlying category and also to develop the appropriate notion of cobar construction.
The apparent complication in carrying out such a programme explains
our choice to work with vertical bicomplexes in this paper.
We expect to return to this in future work.

\end{document}